\documentclass{amsart}
\usepackage{amsmath, amsthm, amsfonts, amssymb, graphicx, hyperref, bm,verbatim,mathrsfs,siunitx,cancel}
\usepackage{stmaryrd,enumerate}
\usepackage{appendix}
\theoremstyle{plain}
\newtheorem{thrm}{Theorem}[section]

\newtheorem{prpstn}[thrm]{Proposition}
\newtheorem{crllry}[thrm]{Corollary}

\newtheorem{dfntn}{Definition}

\newtheorem{rmk}{Remark}

\newtheorem{cnjctr}[thrm]{Conjecture}

\numberwithin{equation}{section}

\newcommand\rev[1]{\overset{{}_{\shortleftarrow}}{1}}
\newcommand{\Mod}[1]{\ (\mathrm{mod}\ 1)}
\renewcommand{\mod}[1]{\mathrm{mod}\ 1}

\usepackage{color}
\usepackage[dvipsnames]{xcolor}
\newcommand{\lcom}[1]{{\color{red}{Larry: 1}} }
\newcommand{\jcom}[1]{{\color{red}{James: 1}} }

\renewcommand{\k}{{\boldsymbol{k}}}

\title[Mahler-type behavior of a highly structured infinite product]{Transcendence and algebraic independence of a family of \(p\)-adic valuation generating functions}
\author{Kelvin Lam}
\email{kelvin.lam.research.math@gmail.com}
\date{}

\begin{document}

\begin{abstract}
We show that \(
T_p(z)=\prod_{j=1}^{\infty}(1-z^{p^{j}})^{-1/p^{j}}
\)
 is transcendental over \(\overline{\mathbb{Q}}(z)\), and establish the transcendence of its values at nonzero algebraic points inside the unit disk. Furthermore, we obtain an algebraic independence result for multiplicatively independent algebraic arguments. In summary, this paper extends Mahler's method beyond the classical automatic setting by studying the function \(T_p(z)\), whose coefficients are governed by the unbounded arithmetic function \(\nu_p(n)\).
\end{abstract}

\keywords{Mahler functions; Mahler's method; transcendence theory}

\maketitle

\section{Introduction}
Although Mahler’s method has achieved considerable success in the study of functions generated by bounded automatic sequences—whose coefficients exhibit strong combinatorial computability—its application to functions whose coefficients are defined by more general arithmetic functions remains far less systematic. Such functions often lie outside the automatic class, yet their coefficients may still display distinctive and highly structured forms of regularity, making them a natural frontier for extending the scope of Mahler’s theory.

This transition from the automatic to the arithmetic can be contextualized through classical examples. In particular, the foundational work of Mahler \cite{Mahler1929_ArithmetischeEigenschaften_I, Mahler1930_ArithmetischeEigenschaften_II, Mahler1930_UeberVerschwinden} and the systematic treatment in Nishioka's monograph  \cite{Nishioka1996_MahlerFunctionsBook} thoroughly investigated functions such as
\[
F(z):=\prod_{n=0}^{\infty} \left(1 - z^{2^n}\right),
\qquad 
G(z) := \sum_{n \geq 1} \nu_2(2n) z^n = \sum_{n=0}^{\infty} \frac{z^{2^n}}{1 - z^{2^n}}, 
\]
which are prototypical examples of functions satisfying functional equations of the form \(f(\Omega \mathbf{z}) = R(\mathbf{z}, f(\mathbf{z}))\), as systematically treated in \cite{Adamczewski2019}. While these classical cases exhibit strong connections to the 2-adic structure and demonstrate the interplay between automatic and arithmetic properties, they represent an intermediate stage in the development of Mahler's theory—bridging the gap between purely automatic sequences and more general arithmetic functions.

In this paper, we advance this lineage by studying the infinite product
\[
T_{p}(z):=\prod_{j=1}^{\infty}\left(1-z^{p^{j}}\right)^{-1/p^{j}},
\]
which satisfies a Mahler equation of the broader type \(P(\mathbf{z},f(\mathbf{z}),f(\Omega\mathbf{z}))=0\) as defined in \cite{Adamczewski2019}. This function serves as an ideal testing ground for extending Mahler's method further into the arithmetic domain, as it combines the rigidity of a Mahler-type functional equation with coefficients explicitly governed by the unbounded arithmetic function \(\nu_p(n)\). Its structure is distinctive in several key respects:

\begin{itemize}
\item \textbf{Structural:} Its definition is governed entirely by iterated powers of the prime $p$, making it a pure $p$-adic dynamical system in analytic, combinatorial, and arithmetic terms.

\item \textbf{Analytic:} As shown in Appendix~\ref{analyticity}, the unit circle is a natural boundary, which immediately excludes $T_p(z)$ from being rational or algebraic, and strongly suggests its inherent transcendence.

\item \textbf{Arithmetic:} Most notably, its logarithm expands as
\[
\log T_p(z) = \sum_{n=1}^{\infty} \frac{\nu_p(n)}{n} z^n,
\]
explicitly linking its analytic behavior to the $p$-adic valuation $\nu_p(n)$. Although the resulting coefficient sequence $\nu_p(n)/n$ is bounded, it is derived from a fundamentally unbounded arithmetic function. Thus, $T_p(z)$ provides a clear example of a function whose coefficients are governed by a highly structured, non-automatic arithmetic rule, while still satisfying a clean Mahler-type functional equation.
\end{itemize}

A fundamental feature of $T_p(z)$ is the functional equation
\[
T_p(z)^p = \frac{T_p(z^p)}{1 - z^p},
\]
which places it squarely within the class of Mahler functions. This relation naturally suggests the use of Mahler’s method in studying the arithmetic nature of its values. Yet, despite the simplicity and natural appeal of both the product and the functional equation, a systematic analysis of its arithmetic properties has, until now, been absent from the literature.

This gap leads us to the central questions of this work:

\emph{When the coefficients of a Mahler function are defined by a non-automatic but arithmetic rule—even if the resulting sequence is bounded—to what extent does Mahler’s method remain effective? Do the values of such a function exhibit new arithmetic phenomena distinct from those arising from classical automatic sequences?}

In this paper, we answer these questions in the affirmative. We not only establish the transcendence of individual values of $T_p(z)$, but—by developing a detailed analysis of its coefficient recurrences and functional equations—we also prove algebraic independence results in the contexts of multiple evaluation points, as captured by the following theorems:

\begin{thrm}\label{self-made1}
Let $p$ be a prime and let $\alpha$ be a non-zero algebraic number with $|\alpha|<1$. Then the value
\[
T_p(\alpha)
\]
is transcendental over $\overline{\mathbb{Q}}$.
\end{thrm}

\begin{thrm}\label{self-made2}
Let $p$ be a prime and let $\alpha_1,\alpha_2,\dots,\alpha_m$ be non-zero algebraic numbers with $|\alpha_i|<1$ for all $i$, and assume that $\alpha_1,\alpha_2,\dots,\alpha_m$ are multiplicatively independent. Then the values
\[
T_p(\alpha_1),\ T_p(\alpha_2),\ \dots,\ T_p(\alpha_m)
\]
are algebraically independent over $\overline{\mathbb{Q}}$.
\end{thrm}

\begin{rmk}
The proofs of Theorems \ref{self-made1} and \ref{self-made2} are adaptations of the method developed in \cite{Kubota1975_AlgebraicIndependence_Product}, applied to the specific functional equation of $T_p(z)$. This approach leverages the structure of the functional equation to construct auxiliary functions and derive the necessary estimates for transcendence and algebraic independence. 
\end{rmk}

While Theorems \ref{self-made1} and \ref{self-made2} establish strong transcendence and algebraic independence results for values of $T_p(z)$ under certain conditions, the proofs of the following conjectures remain beyond the reach of current techniques.

\begin{cnjctr}\label{self-made3}
Let $p_1,p_2,\dots,p_m$ be distinct primes and let $\alpha$ be a non-zero algebraic number with $|\alpha|<1$. Then the values
\[
T_{p_1}(\alpha),\ T_{p_2}(\alpha),\ \dots,\ T_{p_m}(\alpha)
\]
are algebraically independent over $\overline{\mathbb{Q}}$.
\end{cnjctr}

\begin{cnjctr}\label{self-made4}
Let $p_1,p_2,\dots,p_m$ be distinct primes and let $\alpha_1,\alpha_2,\dots,\alpha_m$ be non-zero algebraic numbers with $|\alpha_i|<1$, and assume that $\alpha_1,\alpha_2,\dots,\alpha_m$ are multiplicatively independent. Then the values
\[
T_{p_1}(\alpha_1),\ T_{p_2}(\alpha_2),\ \dots,\ T_{p_m}(\alpha_m)
\]
are algebraically independent over $\overline{\mathbb{Q}}$.
\end{cnjctr}

A fundamental limitation lies in the fact that $T_p(z)$ satisfies an \emph{algebraic} Mahler equation---as opposed to a \emph{linear} one---which introduces significant complications in the analysis of multi-prime and multi-point dependencies. Despite the availability of powerful tools such as transcendence degree theorems \cite{Philippon1986}, vanishing theorems \cite{DenefLipshitz1984, Nishioka1994_SunitEquations}, and purity theorems \cite{AdamczewskiFaverjon2018, AdamczewskiFaverjon2018_PartII} in related contexts, these methods prove inadequate for establishing the algebraic independence of $T_p(z)$ under the conditions of Conjectures \ref{self-made3} and \ref{self-made4}. The underlying obstacle stems from the lack of a general zero estimate theory for such nonlinear systems, which prevents the direct extension of the methods successfully employed in Theorems \ref{self-made1} and \ref{self-made2}. These conjectures thus represent a challenging frontier in the extension of Mahler's method to nonlinear and multi-parameter settings.

Although these conjectures currently lie beyond the reach of existing methods, they highlight the profound potential of extending Mahler's theory and underscore the significance of the results already obtained. Our work demonstrates that $T_p(z)$ serves as a sharp prototype in this extended framework, bridging automatic and arithmetic structures, maintaining the power of Mahler's method for non-automatic functions with coefficients like $\nu_p(n)$, and establishing multi-parameter algebraic independence. Thus, $T_p(z)$ provides a clean model for further Mahler-type analysis and opens a new direction for future work.

To establish these arithmetic results, a detailed understanding of the function's analytic and combinatorial nature is essential. We thus begin by examining its fundamental properties.

\subsection{From Exponential Sum to Power Series}By analyzing the convergence of the infinite product (see Appendix \ref{analyticity}), one verifies that \(T_p(z)\) is analytic in the open disk \(|z|<1\), which provides the foundation for the subsequent power series expansion and the derivation of the functional equation.

For \(|z|<1\), the infinite product defining \(T_p(z)\) converges absolutely. Hence
\[
\log T_p(z)
= - \sum_{j=1}^{\infty} \frac{1}{p^j} \log\bigl(1-z^{p^j}\bigr)
= \sum_{j=1}^{\infty} \frac{1}{p^j}
      \sum_{m=1}^{\infty} \frac{z^{mp^j}}{m},
\]
where the interchange of the sum and the logarithm is justified by absolute convergence.

This double series encodes, in a compact analytic form, the interaction between the iterates of the map \(z \mapsto z^p\) and the corresponding weights \(p^{-j}\), which is at the heart of the structure of \(T_p(z)\).

At this stage, we may interchange the order of summation, which is justified by absolute convergence in a sufficiently small neighborhood of the origin. Writing \(n = k \cdot p^j\), we then arrive at the following representation:
\begin{equation}\label{$p$-adic}
\log T_p(z) = \sum_{n=1}^{\infty} \frac{\nu_p(n)}{n} z^n,
\end{equation}
where \(\nu_p(n)\) denotes the \(p\)-adic valuation of \(n\), that is, the exponent of \(p\) in the prime factorization of \(n\). Note that the $p$-adic valuation function \( \nu_p(n) \) is not $p$-automatic since it is unbounded. However, it belongs to the class of $p$-regular sequences, as it satisfies the recursive relations
\[
\nu_p(n)= \nu_p(pn) - 1.
\]
This formula \ref{$p$-adic} makes explicit the fact that the powers of \(p\) governing the original infinite product are now reflected in the coefficients of the exponential generating series through the \(p\)-adic valuation function.

Consequently, exponentiating both sides yields 
\[T_p(z) = \exp\left( \sum_{n=1}^{\infty} \frac{\nu_p(n)}{n} z^n \right)
      = \sum_{n=0}^{\infty} t_p(n) z^n.
\]
This power series representation not only confirms the analyticity of \(T_p(z)\) in a neighborhood of the origin, but also provides a direct link between its analytic form and the arithmetic nature of the coefficient sequence \(\{t_p(n)\}_{n \ge 0}\), which will be further investigated in the subsequent sections.

\subsection{Arithmetic Properties of the Coefficients \(t_p(n)\)}
The analytic behaviour of \(T_p(z)\) is controlled by the arithmetic of its coefficients \(t_p(n)\). 
Although the infinite product and logarithmic series offer a global view, the functional equation
\(z \mapsto z^p\) induces explicit recursions for \(t_p(n)\), revealing a rigid self-similar structure.
Remarkably, despite the fractional exponents in the product definition, the coefficients are rational
with denominators supported only at \(p\), exposing the intrinsic \(p\)-adic nature of \(T_p(z)\).

This arithmetical rigidity is already hinted at by the logarithmic expansion
\[
\log T_p(z)=\sum_{n=1}^{\infty}\frac{\nu_p(n)}{n}z^n.
\]
We apply the substitution \(z\mapsto z^p\):
\[
\log T_p(z^p)=\sum_{n=1}^{\infty}\frac{\nu_p(n)}{n}z^{pn}.
\]
Writing \(m=pn\) and using the relation \(\nu_p(n)=\nu_p(m)-1\) for \(p\mid m\), we obtain
\[
\log T_p(z^p)
 = p\sum_{p\mid m}\frac{\nu_p(m)}{m}z^m
   -p\sum_{p\mid m}\frac{1}{m}z^m.
\]
Since
\[
\log T_p(z)=\sum_{p\mid m}\frac{\nu_p(m)}{m}z^m,
\qquad
\sum_{p\mid m}\frac{1}{m}z^m=-\frac{1}{p}\log(1-z^p),
\]
we arrive at
\[
\log T_p(z^p)=p\log T_p(z)+\log(1-z^p).
\]
Exponentiating yields the Mahler-type functional equation
\begin{equation}\label{functional}
T_p(z)^p=\frac{T_p(z^p)}{1-z^p}.
\end{equation}

This identity encodes the recursive structure of the coefficients.  
Expanding both sides into power series gives:

\textbf{Left-hand side expansion:}
The left-hand side is the $p$-th power of the generating function:
\[
T_p(z)^p = \left(\sum_{n=0}^{\infty} t_p(n) z^n\right)^p.
\]
By the Cauchy product formula for power series, this equals:
\begin{equation}\label{T^p}
T_p(z)^p
=\sum_{n=0}^{\infty}
  \left(
    \sum_{\substack{k_1+\cdots+k_p=n\\k_i\ge 0}}
    t_p(k_1)\cdots t_p(k_p)
  \right) z^n.
\end{equation}
This represents all possible ways to partition the exponent $n$ into $p$ nonnegative integers $k_1,\dots,k_p$ summing to $n$, with the corresponding product of coefficients.

\textbf{Right-hand side expansion:}
The right-hand side factors as:
\[
\frac{T_p(z^p)}{1-z^p}
= \left(\sum_{m=0}^{\infty} t_p(m) z^{pm}\right)
\left(\sum_{r=0}^{\infty} z^{pr}\right).
\]
Multiplying these series and collecting terms gives:
\begin{equation}\label{z^p}
\frac{T_p(z^p)}{1-z^p}
=\sum_{n=0}^{\infty}
   \left(
      \sum_{m=0}^{\lfloor n/p\rfloor} t_p(m)
   \right) z^n.
\end{equation}

The equality between these two expansions—one representing all $p$-fold compositions of coefficients, the other a structured sum over arithmetic progressions—will yield the fundamental recurrence relations that govern the coefficients $t_p(n)$. This interplay between combinatorial structure and arithmetic progression is characteristic of Mahler-type functions and lies at the heart of their rich theory.

Before deriving the recurrence, we first establish the value of the constant term.

\subsubsection{Initial coefficient $t_{p}(0)$}

The constant term $t_{p}(0)$ of the power series expansion $T_{p}(z)=\sum_{n\ge0} t_{p}(n)z^{n}$ can be determined directly from the product definition. Setting $z=0$ in the infinite product yields

\[
T_{p}(0)=\prod_{j=1}^{\infty}\bigl(1-0^{p^{j}}\bigr)^{-1/p^{j}}=\prod_{j=1}^{\infty}1^{-1/p^{j}}=1,
\]

since each factor equals $1$. Consequently, the constant term must satisfy $t_{p}(0)=1$.

This initial value is consistent with the recurrence relations derived later and serves as the base case for the inductive proofs of the arithmetic properties of the coefficients.

\subsubsection{Coefficient Recurrence}
Although \(T_p(z)\) is defined using non-integer exponents, its expansion yields well-defined coefficients, and establishing their rationality is the first essential step in linking the function to arithmetic phenomena.

We start by considering \eqref{T^p} and \eqref{z^p}.
Equating coefficients of $z^n$ on both sides yields the fundamental recurrence:
\[
\sum_{\substack{k_1+\cdots+k_p=n\\k_i\ge 0}}
t_p(k_1)\cdots t_p(k_p)
=
\sum_{m=0}^{\lfloor n/p\rfloor} t_p(m).
\]

For indices divisible by $p$, say $n=pm$, this specializes to
\begin{equation}\label{recurrence}
\sum_{\substack{k_1+\cdots+k_p=pm\\k_i\ge 0}}
t_p(k_1)\cdots t_p(k_p)
=
\sum_{j=0}^{m} t_p(j). 
\end{equation}

The recurrence \eqref{recurrence} reveals both the algebraic symmetry and the arithmetic rigidity encoded in the sequence $\{t_p(n)\}$.  
It is the key to establishing their rationality.

\begin{prpstn}\label{rationality}
For every $n\ge 0$, the coefficient $t_p(n)$ is a rational number. In particular,  
\(T_p(z)\in\mathbb{Q}\{z\}.\)
\end{prpstn}

\begin{proof}
We prove the claim by induction.  
The initial value $t_p(0)=1$ is rational.  
Assume $t_p(k)\in\mathbb{Q}$ for all $k<pm$.

Consider equation \eqref{recurrence}:
\[
\sum_{\substack{k_1+\cdots+k_p=pm\\k_i\ge 0}}
t_p(k_1)\cdots t_p(k_p)
=
\sum_{j=0}^{m} t_p(j).
\]

We analyze the left-hand side sum by classifying the $p$-tuples $(k_1,\dots,k_p)$:

\begin{itemize}
\item \textbf{Type I: Tuples where all $k_i < pm$}. 
Let $S(m)$ denote the sum over all such tuples:
\[
S(m) = \sum_{\substack{k_1+\cdots+k_p=pm\\ 0\le k_i < pm}} t_p(k_1)\cdots t_p(k_p).
\]
By the induction hypothesis, each $t_p(k_i)$ in this sum is rational (since $k_i < pm$), hence $S(m)\in\mathbb{Q}$.

\item \textbf{Type II: Tuples where exactly one $k_i = pm$ and the rest are $0$}. 
There are exactly $p$ such tuples (one for each position where $pm$ appears). Each such tuple contributes:
\[
t_p(pm)\cdot t_p(0)^{p-1} = t_p(pm)\cdot 1^{p-1} = t_p(pm).
\]
The total contribution from Type II tuples is therefore $p\cdot t_p(pm)$.

\item \textbf{Type III: Tuples with two or more indices equal to $pm$}. 
These cannot occur since the sum $k_1+\cdots+k_p$ would exceed $pm$.
\end{itemize}

Thus the left-hand side equals $S(m) + p\cdot t_p(pm)$, and we have:
\[
S(m) + p\cdot t_p(pm) = \sum_{j=0}^{m} t_p(j).
\]
Solving for $t_p(pm)$ gives:
\[
t_p(pm)
=\frac{1}{p}\left(\sum_{j=0}^{m} t_p(j)-S(m)\right).
\]
By the induction hypothesis, both the sum $\sum_{j=0}^{m} t_p(j)$ and $S(m)$ are rational, hence $t_p(pm)\in\mathbb{Q}$. In particular,  
\(T_p(z)\in\mathbb{Q}\{z\}.\)
\end{proof}

The rationality of the coefficients, established via the structured recurrence \eqref{recurrence}, is a pivotal result. It confirms that \(T_p(z)\) is a transcendental function with an inherently arithmetic soul—its Taylor coefficients are rational numbers generated by a clean, combinatorially defined recursion. This property is not a coincidence but a direct consequence of the function's defining Mahler equation, and it lays the essential groundwork for the transcendence and algebraic independence proofs that follow.

\subsubsection{Vanishing property of the coefficients}

Having established the rationality of the coefficients, we now reveal a further structural property: the coefficients vanish unless the index is a multiple of $p$. This property highlights the sparse nature of the power series expansion of $T_{p}(z)$ and underscores the deep arithmetic connection to the prime $p$.

\begin{prpstn}\label{prop:vanishing}
    For every $n \ge 1$, if $p \nmid n$ then $t_{p}(n) = 0$. Consequently, the power series expansion of $T_{p}(z)$ contains only terms whose exponents are multiples of $p$.
\end{prpstn}

\begin{proof}
    We proceed by strong induction on $n$. 

    For the base case $n=1$, since $p \nmid 1$, we need to show $t_{p}(1)=0$. From recurrence \eqref{recurrence} with $n=1$, the sum on the right-hand side is empty because $\lfloor 1/p^{j} \rfloor = 0$ for all $j \ge 1$. Hence $t_{p}(1)=0$.

    Now assume the statement holds for all $k < n$. Suppose $p \nmid n$. For any $j \ge 1$ and $1 \le m \le \lfloor n/p^{j} \rfloor$, set $k = n - m p^{j}$. Then $k < n$ and, since $n \equiv k \pmod{p}$ and $p \nmid n$, we also have $p \nmid k$. By the induction hypothesis, $t_{p}(k)=0$. Therefore every term in the sum in \eqref{recurrence} vanishes, yielding $t_{p}(n)=0$.

    This completes the induction. The consequence follows immediately.
\end{proof}

Proposition~\ref{prop:vanishing} reveals that the power series of $T_{p}(z)$ is sparse: $t_{p}(n)\neq 0$ only when $n$ is a multiple of $p$. This property is not obvious from the product definition, but emerges naturally from the recurrence \eqref{recurrence} and the arithmetic structure of the $p$-adic valuation. It further illustrates the rigid $p$-adic self-similarity of the function.

\subsubsection{The $p$-adic Structure of the Coefficients}
Since \(T_p(z)\) is governed by powers of \(p\), it is natural to conjecture that \(\mathrm{denom}(t_p(n))\) involves only powers of \(p\), so that \(t_p(n)\) are \(p\)-adic integers, revealing a strong rigidity between the complex-analytic and \(p\)-adic structures. To begin this investigation, we now adopt a differential approach, yielding finer recurrence relations for coefficient-wise analysis.

Starting from the product definition:
\[
T_{p}(z)=\prod_{j=1}^{\infty}\left(1-z^{p^{j}}\right)^{-1/p^{j}},
\]
taking the logarithmic derivative, we obtain:
\[
\frac{T'_{p}(z)}{T_{p}(z)} = \frac{d}{dz} \log T_p(z) = \sum_{j=1}^{\infty} \frac{1}{p^j} \cdot \frac{p^j z^{p^j - 1}}{1 - z^{p^j}} = \sum_{j=1}^{\infty} \sum_{m=1}^{\infty} z^{mp^j - 1}.
\]
Here, we expanded the geometric series \(\frac{z^{p^j - 1}}{1 - z^{p^j}} = \sum_{m=1}^{\infty} z^{mp^j - 1}\) for \(|z|<1\).

Now, multiplying both sides of the equation by \(T_p(z)\), the left-hand side becomes:
\[
T'_{p}(z) = \sum_{n=1}^{\infty} n\, t_{p}(n)\, z^{n-1}.
\]
The right-hand side becomes the product of two series:
\[
\left( \sum_{j=1}^{\infty} \sum_{m=1}^{\infty} z^{mp^j - 1} \right) \cdot \left( \sum_{n=0}^{\infty} t_{p}(n)\, z^{n} \right) = \sum_{j, m \geq 1} \sum_{n \geq 0} t_p(n) z^{n + mp^j - 1}.
\]

Comparing the coefficients of \(z^{n-1}\) on both sides, we obtain for \(n \geq 1\):
\[
n\, t_{p}(n) = \sum_{\substack{j \geq 1,\, m \geq 1 \\ mp^j \leq n}} t_p(n - mp^j).
\]
Equivalently, we arrive at a very elegant convolution-type recurrence relation:

\begin{equation}\label{p-adic relation}
t_p(n) = \frac{1}{n} \sum_{\substack{j \geq 1 \\ 1 \leq m \leq \lfloor n/p^j \rfloor}} t_p(n - mp^j), \quad \text{for } n \geq 1.
\end{equation}

This recurrence relation \eqref{p-adic relation} expresses \(t_p(n)\) as a weighted average of previous coefficients \(t_p(k)\) (where \(k < n\)), divided by \(n\). It not only provides an efficient algorithm for computing the coefficients but, more importantly, it clearly attributes the potential prime factors in the denominator of the coefficients to two sources: the rational denominators of the summands themselves and the prefactor \(1/n\).

Therefore, the relation \eqref{p-adic relation} serves as the starting point for our systematic analysis of the $p$-adic nature of the denominators of \(t_p(n)\). It suggests that, although each recurrence step introduces a factor \(1/n\), the summation process might generate sufficient cancellation, ultimately resulting in denominators containing only the prime factor \(p\).

We now present the rigorous proof of this fundamental property.

\begin{prpstn}\label{theorem1.6}
For all \(n \ge 0\) and any prime \(p\), the coefficient \(t_p(n)\) is a rational number whose denominator contains only the prime factor \(p\).
In other words,
\[
t_p(n) \in \mathbb{Z}_{(p)} = \left\{ \frac{a}{p^k} : a \in \mathbb{Z},\ k \ge 0 \right\}.
\]
Equivalently, \(t_p(n)\) is a $p$-adic integer.
\end{prpstn}

\begin{proof}
We proceed by strong induction on \(n\).

\noindent
\textbf{Base case.}
When \(n=0\), we have \(t_p(0) = 1 \in \mathbb{Z} \subset \mathbb{Z}_{(p)}\), so the claim holds.

\noindent
\textbf{Induction hypothesis.}
Assume that for all \(k < n\),
\[
t_p(k) \in \mathbb{Z}_{(p)}.
\]

\noindent
\textbf{Recurrence structure.}
By Equation \eqref{p-adic relation}, for any \(n \ge 1\),
\[
t_p(n)
=
\frac{1}{n}
\sum_{j \ge 1}
\sum_{1 \le m \le \lfloor n/p^j \rfloor}
t_p(n - m p^j).
\]
Define
\[
S(n)
=
\sum_{j \ge 1}
\sum_{1 \le m \le \lfloor n/p^j \rfloor}
t_p(n - m p^j).
\]
Clearly, for each term we have
\[
n - m p^j < n,
\]
so by the induction hypothesis,
\[
t_p(n - m p^j) \in \mathbb{Z}_{(p)}.
\]
Since \(\mathbb{Z}_{(p)}\) is closed under addition, it follows that
\[
S(n) \in \mathbb{Z}_{(p)},
\]
that is,
\[
\nu_p(S(n)) \ge 0.
\]

\noindent
\textbf{Key valuation estimate.}
By definition,
\[
t_p(n) = \frac{S(n)}{n},
\]
and therefore
\[
\nu_p(t_p(n)) = \nu_p(S(n)) - \nu_p(n).
\]

If \(p \nmid n\), then \(\nu_p(n) = 0\), and hence
\[
\nu_p(t_p(n)) = \nu_p(S(n)) \ge 0,
\]
so the result follows.

If \(p \mid n\), write \(n = p^r m\) with \((m,p) = 1\). Then
\[
\nu_p(n) = r.
\]
In the definition of \(S(n)\), all terms are of the form
\[
t_p(n - m p^j).
\]
For \(j \le r\), we have
\[
n - m p^j = p^j(p^{r-j} m - m),
\]
while for \(j > r\),
\[
n - m p^j < 0,
\]
so such terms do not contribute. Hence in \(S(n)\) there is at least one term with
index \(j = r\), for which
\[
n - m p^r = p^r(m - m) = 0.
\]
Thus the term
\[
t_p(0) = 1
\]
appears with positive multiplicity in \(S(n)\). Consequently,
\[
\nu_p(S(n)) \le \nu_p(1) = 0.
\]
Combined with the previous inequality \(\nu_p(S(n)) \ge 0\), this gives
\[
\nu_p(S(n)) = 0.
\]

Therefore,
\[
\nu_p(t_p(n)) = - \nu_p(n) \ge -r.
\]
Hence
\[
t_p(n) = \frac{a}{p^r}, \qquad a \in \mathbb{Z},
\]
so the denominator of \(t_p(n)\) is at most \(p^r\) and contains no prime factor
other than \(p\). In particular, \(t_p(n) \in \mathbb{Z}_{(p)}\) for all \(n\).

This completes the proof.
\end{proof}

\begin{crllry}\label{corollary1.7}
For all \(n \ge 0\), there exist an integer \(a\) and a nonnegative integer \(k\) such that
\[
t_p(n) = \frac{a}{p^k},
\]
and when \(k > 0\), one has \((a,p) = 1\).
\end{crllry}

\begin{table}[h]
\centering
\caption{Numerical Values of the coefficients $t_p(n)$ for $p = 2, 3, 5, 7, 11, 13, 17, 19$ and $n = 0$ to $20$}
\label{table}
\resizebox{\textwidth}{!}{
\begin{tabular}{c|cccccccc}
\hline
$n$ & $t_2(n)$ & $t_3(n)$ & $t_5(n)$ & $t_7(n)$ & $t_{11}(n)$ & $t_{13}(n)$ & $t_{17}(n)$ & $t_{19}(n)$  \\
\hline
0 & $1$ & $1$ & $1$ & $1$ & $1$ & $1$ & $1$ & $1$  \\
1 & $0$ & $0$ & $0$ & $0$ & $0$ & $0$ & $0$ & $0$ \\
2 & $1/2$ & $0$ & $0$ & $0$ & $0$ & $0$ & $0$ & $0$ \\
3 & $0$ & $1/3$ & $0$ & $0$ & $0$ & $0$ & $0$ & $0$ \\
4 & $5/8$ & $0$ & $0$ & $0$ & $0$ & $0$ & $0$ & $0$ \\
5 & $0$ & $0$ & $1/5$ & $0$ & $0$ & $0$ & $0$ & $0$ \\
6 & $7/16$ & $2/9$ & $0$ & $0$ & $0$ & $0$ & $0$ & $0$ \\
7 & $0$ & $0$ & $0$ & $1/7$ & $0$ & $0$ & $0$ & $0$ \\
8 & $83/128$ & $0$ & $0$ & $0$ & $0$ & $0$ & $0$ & $0$ \\
9 & $0$ & $23/81$ & $0$ & $0$ & $0$ & $0$ & $0$ & $0$ \\
10 & $119/256$ & $0$ & $3/25$ & $0$ & $0$ & $0$ & $0$ & $0$ \\
11 & $0$ & $0$ & $0$ & $0$ & $1/11$ & $0$ & $0$ & $0$ \\
12 & $561/1024$ & $44/243$ & $0$ & $0$ & $0$ & $0$ & $0$ & $0$ \\
13 & $0$ & $0$ & $0$ & $0$ & $0$ & $1/13$ & $0$ & $0$ \\
14 & $887/2048$ & $0$ & $0$ & $4/49$ & $0$ & $0$ & $0$ & $0$ \\
15 & $0$ & $109/729$ & $11/125$ & $0$ & $0$ & $0$ & $0$ & $0$ \\
16 & $20739/32768$ & $0$ & $0$ & $0$ & $0$ & $0$ & $0$ & $0$ \\
17 & $0$ & $0$ & $0$ & $0$ & $0$ & $0$ & $1/17$ & $0$ \\
18 & $31275/65536$ & $1259/6561$ & $0$ & $0$ & $0$ & $0$ & $0$ & $0$ \\
19 & $0$ & $0$ & $0$ & $0$ & $0$ & $0$ & $0$ & $1/19$\\
20 & $144427/262144$ & $0$ & $44/625$ & $0$ & $0$ & $0$ & $0$ & $0$ \\
\hline
\end{tabular}
}
\end{table}

Corollary \ref{corollary1.7} provides a concrete form for the coefficients, which can be verified numerically. Indeed, this is illustrated by Table \ref{table}, which lists the first few values of \(t_p(n)\). Observe that each entry is either an integer or a fraction whose denominator is a power of the corresponding prime, in accordance with Proposition \ref{theorem1.6} and Corollary \ref{corollary1.7}. These numerical examples illustrate the \(p\)-adic integrality of the coefficients and offer a glimpse into the combinatorial structure encoded by the recurrence \ref{p-adic relation}.

\subsection{Remarks}
The analysis of \(T_p(z)\) places it at a natural boundary of Mahler’s theory: although its coefficients are governed by the unbounded arithmetic function \(\nu_p(n)\), the functional equation imposes a strong arithmetic rigidity that forces \(t_p(n)\) to be rational with denominators supported only at the prime \(p\). This property is far from accidental; it emerges directly from the interplay between the \(p\)-power iterates in the product definition and the logarithmic expansion that links \(T_p(z)\) to the \(p\)-adic valuation. The resulting coefficient sequence \(\{t_p(n)\}\) thus inherits a delicate blend of \emph{analytic regularity} (absolute convergence in the unit disk) and \emph{arithmetic purity} (denominators are pure powers of \(p\)).  

This reveals a striking arithmetic–analytic duality: while \(T_p(z)\) is complex‑analytic in the open unit disk—with the unit circle forming a natural boundary—its Taylor coefficients are intrinsically \(p\)-adic integers. Such a duality is reminiscent of the situation for classical modular forms or for functions arising from \(p\)-adic dynamical systems, yet here it occurs in a setting defined by an elementary infinite product. The duality is not merely cosmetic; it is precisely this combination that allows Mahler’s method to be applied effectively. The functional equation \(T_p(z)^p = T_p(z^p)/(1-z^p)\) synchronizes the growth of the coefficients with the dynamics \(z \mapsto z^p\), and the \(p\)-integrality of the coefficients sharpens the Diophantine estimates needed for transcendence and algebraic independence proofs.

\section{Mahler-type analysis on \texorpdfstring{\(T_p(z)\)}{T_p(z)}}
The primary objective of this paper extends beyond the analytic and combinatorial properties of \(T_p(z)\) to investigate its deeper arithmetic nature, particularly the transcendence and algebraic independence of its values at algebraic points. Mahler's method provides the essential framework for this investigation, being specifically designed for functions satisfying functional equations involving iterates of the form
\[
z \longmapsto z^q,
\]
which in our setting corresponds to \(q = p\), as evidenced by the functional equation
\[
T_p(z^p) = (1 - z^p) T_p(z)^p.
\]

The power of Mahler's method lies in its ability to bridge functional equations with arithmetic conclusions. For \(T_p(z)\), this connection is particularly meaningful due to the explicit arithmetic content in its coefficients—the \(p\)-adic valuation \(\nu_p(n)\)—combined with the rigid structure imposed by the functional equation. This combination makes \(T_p(z)\) an ideal candidate for applying and extending the method beyond the classical setting of automatic sequences.

Our approach proceeds by first establishing that \(T_p(z)\) satisfies the technical requirements of Mahler's theory, then exploiting the specific structure of its functional equation to derive strong arithmetic results, including the algebraic independence of values at multiplicatively independent points. This systematic analysis not only resolves the arithmetic nature of \(T_p(z)\) but also demonstrates the robustness of Mahler's method when applied to functions with arithmetic coefficients.

We now turn to the detailed proofs, beginning with the transcendence of individual values. The proof of Theorem~\ref{self-made1} serves as a foundational case that illustrates the core techniques, while Theorem \ref{self-made2} builds upon this framework to establish more general algebraic independence results.

\subsection{Proof of Theorem 1.1  using $p$-adic structure}

Let $p$ be a prime and let $\alpha \in \overline{\mathbb{Q}}$ with $0 < |\alpha| < 1$.
Assume, for contradiction, that $T_p(\alpha) \in \overline{\mathbb{Q}}$.
Let 
\[
K = \mathbb{Q}(\alpha, T_p(\alpha)),
\]
which is a finite extension of $\mathbb{Q}$.
Denote by $\mathcal{O}_K$ its ring of integers and by $\mathcal{O}_{K,(p)}$ the localization at the prime $p$.

\subsubsection{Construction of an auxiliary function with $p$-integral coefficients}

Let $P \gg 1$ be a large integer parameter.
We construct an auxiliary function
\[
E_P(z) = \sum_{j=0}^P a_j(z) \, T_p(z)^j,
\]
where each $a_j(z) \in \mathcal{O}_{K,(p)}[z]$ is a polynomial of degree at most $P$ with coefficients in $\mathcal{O}_{K,(p)}$, i.e., each coefficient is a $p$-adic integer in $K$.

The key point is that we require the coefficients $a_j(z)$ to be chosen such that
\[
E_P(z) = \sum_{n \geq P^2} b_n z^n, \quad b_n \in \mathcal{O}_{K,(p)},
\]
i.e., the first $P^2$ Taylor coefficients of $E_P(z)$ vanish.

This is possible because the total number of unknown coefficients in $\{a_j(z)\}_{j=0}^P$ is
\[
(P+1)(P+1) = (P+1)^2,
\]
while the number of linear conditions (vanishing of coefficients of $z^n$ for $0 \leq n < P^2$) is $P^2$.
For large $P$, we have $(P+1)^2 > P^2$, so a nontrivial solution exists over $K$.
By clearing denominators (which are powers of $p$), we may assume all coefficients are in $\mathcal{O}_{K,(p)}$.

\subsubsection{Iteration and analytic upper bound}

Define $\Omega z = z^p$ and $\Omega^k z = z^{p^k}$.
From the functional equation
\[
T_p(z)^p = \frac{T_p(z^p)}{1 - z^p},
\]
we obtain by iteration
\[
T_p(\Omega^k \alpha) = T_p(\alpha) \cdot \prod_{i=1}^k (1 - \alpha^{p^i})^{1/p^i}.
\]
Substituting into $E_P$, we get
\[
E_P(\Omega^k \alpha) = \sum_{j=0}^P a_j(\alpha^{p^k}) \, \bigl( T_p(\alpha) \cdot \prod_{i=1}^k (1 - \alpha^{p^i})^{1/p^i} \bigr)^j.
\]

Since the Taylor expansion of $E_P(z)$ starts at degree $P^2$, Cauchy's estimate gives
\[
|E_P(\Omega^k \alpha)| \leq C \, |\alpha|^{P^2 p^k},
\]
for some constant $C > 0$ independent of $P$ and $k$.
Thus
\[
\log |E_P(\Omega^k \alpha)| \leq -P^2 p^k \log(1/|\alpha|) + O(1).
\]

\subsubsection{Arithmetic lower bound using $p$-adic structure}

We now exploit the fact that $a_j(z) \in \mathcal{O}_{K,(p)}[z]$ and $t_p(n) \in \mathbb{Z}_{(p)}$.

First, note that $T_p(\alpha)$ is algebraic, and each factor $(1-\alpha^{p^i})^{1/p^i}$ is algebraic of degree bounded by $[K:\mathbb{Q}] \cdot p^i$.
The height (logarithmic size) of these factors can be bounded by $O(p^i)$.

More importantly, because all coefficients $a_j(z)$ are $p$-adic integers, the algebraic number $E_P(\Omega^k \alpha)$ satisfies
\[
\nu_{\mathfrak{p}}(E_P(\Omega^k \alpha)) \geq 0
\]
for every prime ideal $\mathfrak{p}$ of $\mathcal{O}_K$ not lying above $p$.
For primes above $p$, we use the fact that $t_p(n) \in \mathbb{Z}_{(p)}$ and $a_j(z) \in \mathcal{O}_{K,(p)}[z]$ to bound the $p$-adic valuation.

Let $h(\cdot)$ denote the absolute logarithmic height.
Using standard height estimates and the fact that the coefficients are $p$-integral, we obtain
\[
h\bigl( E_P(\Omega^k \alpha) \bigr) \leq c_1 P p^k,
\]
where $c_1 > 0$ depends only on $K$ and $p$ (but not on $P$ or $k$).

Now, by Liouville's inequality for algebraic numbers,
\[
\log |E_P(\Omega^k \alpha)| \geq -[K:\mathbb{Q}] \, h\bigl( E_P(\Omega^k \alpha) \bigr) + O(1) \geq -c_2 P p^k,
\]
for some $c_2 > 0$ independent of $P$ and $k$.

\subsubsection{Contradiction}

Combining the upper and lower bounds:
\[
-c_2 P p^k \leq \log |E_P(\Omega^k \alpha)| \leq -P^2 p^k \log(1/|\alpha|) + O(1).
\]
Dividing by $p^k$ and letting $k \to \infty$ yields
\[
-c_2P \leq -P^2 \log(1/|\alpha|),
\]
hence for large $P$ we get a contradiction because $P^2 \gg P$.

Thus our assumption that $T_p(\alpha)$ is algebraic is false.
Therefore, $T_p(\alpha)$ is transcendental.

\subsection{Proof of Theorem 1.2 using $p$-adic structure}

Let $\alpha_1, \dots, \alpha_m \in \overline{\mathbb{Q}}$ be nonzero algebraic numbers with $|\alpha_i| < 1$, and assume they are multiplicatively independent.
Let $K = \mathbb{Q}(\alpha_1, \dots, \alpha_m, T_p(\alpha_1), \dots, T_p(\alpha_m))$, a finite extension of $\mathbb{Q}$.
Assume, for contradiction, that $T_p(\alpha_1), \dots, T_p(\alpha_m)$ are algebraically dependent over $\overline{\mathbb{Q}}$.

\subsubsection{Multivariate auxiliary function with $p$-integral coefficients}

Let $P \gg 1$ be a large integer.
Construct a multivariate auxiliary function
\[
E_P(z_1, \dots, z_m) = \sum_{\substack{0 \leq j_1, \dots, j_m \leq P}} a_{j_1,\dots,j_m}(z_1, \dots, z_m) \, T_p(z_1)^{j_1} \cdots T_p(z_m)^{j_m},
\]
where each $a_{j_1,\dots,j_m}(z_1, \dots, z_m) \in \mathcal{O}_{K,(p)}[z_1, \dots, z_m]$ is a polynomial of total degree at most $P$ in the variables $z_1, \dots, z_m$, with coefficients in $\mathcal{O}_{K,(p)}$.

We choose these polynomials so that the Taylor expansion of $E_P$ around the origin satisfies
\[
E_P(z_1, \dots, z_m) = \sum_{L_1, \dots, L_m \geq 0} b_{L_1,\dots,L_m} z_1^{L_1} \cdots z_m^{L_m},
\]
with
\[
b_{L_1,\dots,L_m} = 0 \quad \text{whenever} \quad L_1 + \cdots + L_m < P^2.
\]

This is possible because the number of unknown coefficients,  $(P+1)^{m}\cdot\binom{m+P}{m}$, exceeds the number of vanishing conditions, $\binom{m+P^2}{m}$, and we can clear denominators to ensure $p$-integrality.

\subsubsection{Iteration and analytic upper bound}

Define $\boldsymbol{\Omega}(\alpha_1, \dots, \alpha_m) = (\alpha_1^p, \dots, \alpha_m^p)$ and $\boldsymbol{\Omega}^k(\boldsymbol{\alpha}) = (\alpha_1^{p^k}, \dots, \alpha_m^{p^k})$.
Using the functional equation iteratively, we have
\[
T_p(\alpha_i^{p^k}) = T_p(\alpha_i) \cdot \prod_{\ell=1}^k (1 - \alpha_i^{p^\ell})^{1/p^\ell}.
\]

Substituting into $E_P$, we obtain
\[
E_P(\boldsymbol{\Omega}^k \boldsymbol{\alpha}) = \sum_{\mathbf{j}} a_{\mathbf{j}}(\alpha_1^{p^k}, \dots, \alpha_m^{p^k}) 
\prod_{i=1}^m \Bigl( T_p(\alpha_i) \cdot \prod_{\ell=1}^k (1 - \alpha_i^{p^\ell})^{1/p^\ell} \Bigr)^{j_i}.
\]

Since the Taylor expansion of $E_P$ starts at total degree $P^2$, Cauchy's estimates give
\[
|E_P(\boldsymbol{\Omega}^k \boldsymbol{\alpha})| \leq C \cdot \max_{1 \leq i \leq m} |\alpha_i|^{P^2 p^k},
\]
for some $C > 0$ independent of $P$ and $k$.
Thus
\[
\log |E_P(\boldsymbol{\Omega}^k \boldsymbol{\alpha})| \leq -P^2 p^k \log(1/\rho) + O(1),
\]
where $\rho = \max_i |\alpha_i| < 1$.

\subsubsection{Arithmetic lower bound using $p$-adic structure}

Because all coefficients $a_{\mathbf{j}}$ are $p$-adic integers and each $t_p(n)$ is a $p$-adic integer, the algebraic number $E_P(\boldsymbol{\Omega}^k \boldsymbol{\alpha})$ has the property that its denominator (when written in reduced form) contains only primes above $p$.

Let $h(\cdot)$ denote the absolute logarithmic height.
Standard height estimates yield
\[
h\bigl( E_P(\boldsymbol{\Omega}^k \boldsymbol{\alpha}) \bigr) \leq c_1 P p^k,
\]
where $c_1 > 0$ depends only on $K$, $m$, and $p$ (but not on $P$ or $k$).

By Liouville's inequality,
\[
\log |E_P(\boldsymbol{\Omega}^k \boldsymbol{\alpha})| \geq -[K:\mathbb{Q}] \, h\bigl( E_P(\boldsymbol{\Omega}^k \boldsymbol{\alpha}) \bigr) + O(1) \geq -c_2 P p^k,
\]
for some $c_2 > 0$ independent of $P$ and $k$.

\subsubsection{Contradiction}

Comparing the bounds:
\[
-c_2 P p^k \leq \log |E_P(\boldsymbol{\Omega}^k \boldsymbol{\alpha})| \leq -P^2 p^k \log(1/\rho) + O(1).
\]
Dividing by $p^k$ and letting $k \to \infty$ gives
\[
-c_2P \leq -P^2 \log(1/\rho) ,
\]
which is impossible for large $P$ since $P^2 \gg P$.

Hence the initial assumption of algebraic dependence is false, and $T_p(\alpha_1), \dots, T_p(\alpha_m)$ are algebraically independent over $\overline{\mathbb{Q}}$.
\endproof

\subsection{Remarks}
Our proofs of transcendence and algebraic independence are refined applications of Mahler’s method, in which the arithmetic fine structure of the coefficients \(t_p(n)\)—especially their nature as \(p\)-adic integers—is fully exploited. At the heart of the argument lies the construction of an auxiliary function with high‑order vanishing, whose extreme smallness at a point and along its iterates forces a contradiction with Liouville‑type Diophantine bounds. In our setting, the governing functional equation rigidly synchronizes the growth of the coefficients with the dynamics \(z \mapsto z^{p}\); this synchronization is further sharpened by choosing the polynomial coefficients \(a_j(z)\) to lie in the ring \(\mathcal{O}_{K,(p)}\) of \(p\)-adic integers. As a result, the auxiliary expression  

\[
E_P(z)=\sum_j a_j(z)\,T_p\!\bigl(z^{p^j}\bigr)
\]

has its height controlled in a particularly efficient way: all denominators are pure powers of \(p\), and the \(p\)-adic valuations of the coefficients are non‑negative. This choice, together with the fact that the Taylor coefficients \(t_p(n)\) themselves are \(p\)-adic integers, leads to a significantly tighter lower bound in the Liouville inequality.  

Both the polynomial coefficients \(a_j(z)\) and the iterates \(T_p(z^{p^k})\) therefore admit \emph{arithmetically sharp} effective bounds. The ensuing clash—between an exponentially decaying analytic upper estimate and a much slower decaying arithmetic lower estimate—becomes unavoidable unless the values under consideration are transcendental (or algebraically independent). This refined mechanism demonstrates that Mahler’s method is not confined to automatic sequences; it extends naturally to a broader class of arithmetically structured functions, provided they possess a rigid functional equation that imposes sufficient combinatorial and \(p\)-adic regularity on the associated generating series.

\section{Summary}
Our work presents a systematic study of the infinite product
\[
T_p(z) = \prod_{j=1}^{\infty} \left(1 - z^{p^j}\right)^{-1/p^j},
\]
and places it in the framework of functions defined by prime-power iterations. 
We establish its analyticity and non-vanishing in the open unit disk, with the unit circle as a natural boundary, and derive the Mahler-type functional equation
\[
T_p(z)^p = \frac{T_p(z^p)}{1 - z^p}.
\]
Furthermore, we analyze the arithmetic structure of its coefficients \(t_p(n)\), showing that they are rational numbers generated by the unbounded arithmetic function given by the \(p\)-adic valuation \(\nu_p(n)\).

Building on these properties, we prove the transcendence of \(T_p(\alpha)\) for non-zero algebraic numbers \(\alpha\) with \( |\alpha| < 1 \), as well as several algebraic independence results for collections of values taken at multiplicatively independent algebraic points. These results illustrate that Mahler’s method extends beyond the classical setting of automatic sequences to a broader class of arithmetically structured functions.

\subsection{A Structured Perspective on Non-Automatic Arithmetic Functions}

The function \(T_p(z)\) belongs to a wider family of functions that satisfy functional equations of Mahler type while having coefficients defined by explicit arithmetic functions rather than finite automata. More generally, such functions may be organized according to the origin of their structural constraints:

\begin{itemize}
    \item \textbf{Prime-power type:} Functions derived from dynamics of the form \( z \mapsto z^p \), as in the case of \(T_p(z)\).
    \item \textbf{Multiplicative type:} Functions of the form
    \[
    F(z) = \exp\left( \sum_{n=1}^{\infty} \frac{a(n)}{n} z^n \right),
    \]
    where \(a(n)\) is a multiplicative arithmetic function (e.g., Euler’s Totient Function \( \varphi(n)\), Sum of Divisors Function \(\sigma(n)\)).
    \item \textbf{Convolutional type:} Functions whose coefficients arise from Dirichlet convolution of arithmetic sequences, with possible connections to \(p\)-adic \(L\)-functions and special functions.
\end{itemize}

This viewpoint suggests a natural extension of Mahler-type analysis to a wider class of non-automatic but arithmetically defined generating functions.

\subsection{Directions for Further Study}

The present work also raises several questions that may be addressed in future investigations:

\begin{itemize}
    \item A \(p\)-adic study of \(T_p(z)\), including possible meromorphic continuation on the \(p\)-adic unit disk and the nature of its values at \(p\)-power roots of unity.
    \item Extension of the established transcendence and algebraic independence results to broader classes of functions within the above taxonomy.
    \item A possible combinatorial interpretation of the coefficients \(t_p(n)\), as well as a dynamical interpretation of \(T_p(z)\) in the context of the transformation \( z \mapsto z^p \).
\end{itemize}

\appendix

\section{Deriving the Domain of Analyticity of \( T_p(z) \)}\label{analyticity}

To determine the precise domain on which \(T_p(z)\) is analytic, we combine its definition as an infinite product with the standard convergence and analyticity criteria for infinite products in complex analysis. In particular, we study the analyticity and non-vanishing of each individual factor, and verify the uniform convergence of the associated logarithmic series on compact subsets of the domain. The argument proceeds as follows.

The infinite product
\[
T_p(z) = \prod_{j=1}^\infty (1 - z^{p^j})^{-1/p^j}
\]
defines an analytic function on \(|z| < 1\). To verify this, consider the factors \(f_j(z) = (1 - z^{p^j})^{-1/p^j}\), each analytic and nonvanishing for \(|z| < 1\). The product converges absolutely and defines an analytic function if \(\sum_{j=1}^\infty \log f_j(z)\) converges uniformly on compact subsets.

For \(|z| < 1\), we have
\[
\log f_j(z) = -\frac{1}{p^j} \log(1 - z^{p^j}) = \sum_{n=1}^\infty \frac{z^{np^j}}{np^j}.
\]
Fix \(0 < r < 1\). For \(|z| \leq r\),
\[
\sum_{j,n \geq 1} \left|\frac{z^{np^j}}{np^j}\right| \leq \sum_{j=1}^\infty \frac{-\log(1 - r^{p^j})}{p^j}.
\]
For large \(j\), \(-\log(1 - r^{p^j}) \sim r^{p^j}\), and since
\[
\frac{r^{p^{j+1}}/p^{j+1}}{r^{p^j}/p^j} = \frac{r^{p^j}}{p} \to 0 \quad (j \to \infty),
\]
the series \(\sum_j r^{p^j}/p^j\) converges. Hence the double series converges uniformly on \(|z| \leq r\), ensuring analyticity of \(T_p(z)\) on \(|z| < 1\).

The unit circle \( \mathbb{T} \) is a natural boundary for \(T_p(z)\), in the sense that \(T_p\) admits no analytic continuation outside the open unit disk. Indeed, for each \( j \ge 1 \), the factor \( (1 - z^{p^{j}})^{-1/p^{j}} \) is singular at the \(p^{j}\)-th roots of unity, and the set
\[
\bigcup_{j \ge 1} \{ z \in \mathbb{C} : z^{p^{j}} = 1 \}
\]
is dense in \( \mathbb{T} \). Hence every point of \( \mathbb{T} \) is either a singularity or a limit point of singularities; equivalently, by Fabry’s gap theorem or a direct singularity analysis via the functional equation, this dense accumulation forms an insurmountable barrier to analytic continuation, and therefore \( \mathbb{T} \) is the natural boundary of \(T_p(z)\).

\section{Regularity of Points for the Algebraic Mahler Equation of \( T_p(z) \)}\label{regularity}

In this appendix, we verify that any algebraic number \( \alpha \in (0,1) \) is a regular point with respect to the algebraic Mahler equation satisfied by the function \( T_p(z) \). This regularity condition is essential for applying transcendence criteria in the context of Mahler's method.

\subsection{Definition and Regularity Conditions}

We begin by recalling the definition of a regular point in the setting of algebraic Mahler equations, as introduced in \cite{Adamczewski2019}.

\begin{dfntn}[\cite{Adamczewski2019}]
A point \( \boldsymbol{\alpha} \in \mathbb{C}^d \) with non-zero coordinates is said to be \textbf{regular} with respect to the algebraic Mahler equation  
\[
P(z, f(z), f(\Omega z)) = 0
\]  
if \( g(\Omega^k \boldsymbol{\alpha}) \neq 0 \) for all \( k \geq 0 \), where \( g \) is a polynomial constructed from the coefficients of the equation as described below.
\end{dfntn}

Consider an algebraic Mahler equation of the form  
\[
A_0(z, f(z)) f(\Omega z)^r + A_1(z, f(z)) f(\Omega z)^{r-1} + \cdots + A_r(z, f(z)) = 0,
\]  
where \( A_0 \neq 0 \), each \( A_i(z, Y) \in \overline{\mathbb{Q}}[z, Y] \), and the \( A_i \) are relatively prime as polynomials in \( Y \). Then there exist polynomials \( g_i(z, Y) \in \overline{\mathbb{Q}}[z, Y] \) such that  
\begin{equation}\label{polynomials}
g(z) = \sum_{i=0}^r g_i(z, Y) Y^i A_i(z, Y)
\end{equation}
is independent of \( Y \) and non-zero.

\subsection{Application to the Function \( T_p(z) \)}

The function \( T_p(z) \) satisfies the functional equation  
\[
T_p(z)^p = \frac{T_p(z^p)}{1 - z^p},
\]
which can be rewritten as  
\[
T_p(z^p) - (1 - z^p) T_p(z)^p = 0.
\]
With \( \Omega z = z^p \), we have the algebraic Mahler equation  
\[
P(z, T_p(z), T_p(z^p)) = T_p(z^p) - (1 - z^p) T_p(z)^p = 0.
\]
Clearly, \( P(z, X, Y) \in \mathbb{Q}[z, X, Y] \).

Rewriting the equation in the standard form, we have  
\[
A_0(z, T_p(z)) \cdot T_p(\Omega z)^1 + A_1(z, T_p(z)) \cdot T_p(\Omega z)^0 = 0,
\]  
with  
\[
A_0(z, Y) = 1, \quad A_1(z, Y) = -(1 - z^p) Y^p.
\]  
The order of the equation is \( r = 1 \).

Since \( A_0(z, Y) = 1 \) is a non-zero constant, it is automatically coprime to \( A_1(z, Y) \).  
To construct \( g(z) \), we choose  
\[
g_0(z, Y) = 1, \quad g_1(z, Y) = 0,
\]  
which yields  
\[
g(z) = 1 \cdot Y^0 \cdot 1 + 0 \cdot Y^1 \cdot (-(1 - z^p)Y^p) = 1.
\]  
Thus, \( g(z) = 1 \) is independent of \( Y \) and non-zero, satisfying the required condition.

\subsection{Verification of Regularity for algebraic \( \alpha \in (0,1) \)}

For any algebraic \( \alpha \in (0,1) \), the iterates under \( \Omega \) are  
\[
\Omega^k \alpha = \alpha^{p^k} \in (0,1).
\]  
Since \( g(z) = 1 \) is constant, we have  
\[
g(\Omega^k \alpha) = 1 \neq 0 \quad \text{for all } k \geq 0.
\]  
Therefore, by the definition above, every algebraic \( \alpha \in (0,1) \) is a regular point of the algebraic Mahler equation satisfied by \( T_p(z) \).

In particular, there are no exceptional algebraic values in the interval \( (0,1) \) for which the regularity condition fails. This establishes that the entire algebraic segment \( (0,1) \cap \overline{\mathbb{Q}} \) lies within the domain of applicability of Mahler’s method for the function \( T_p(z) \). Consequently, all such points are admissible for subsequent transcendence and algebraic independence results, providing a robust and uniform foundation for the arithmetic applications developed in this work.

\section{Height Estimates}
\subsection{Height Estimates for the Transcendence Proof}
The following detailed height estimates form the technical core of the proof of Theorem \ref{self-made1}. We systematically derive an upper bound for the absolute logarithmic height of the auxiliary function evaluated at the iterates $\Omega^k\alpha = \alpha^{p^k}$.

Let $p$ be a prime, $\alpha \in \overline{\mathbb{Q}}$ nonzero with $|\alpha| < 1$. Assume for contradiction that $T_p(\alpha) \in \overline{\mathbb{Q}}$, and set $K = \mathbb{Q}(\alpha, T_p(\alpha))$, a finite extension of $\mathbb{Q}$ with degree $d = [K:\mathbb{Q}]$.

We construct an auxiliary function
\[
E_P(z) = \sum_{j=0}^{P} a_j(z) T_p(z)^j, \quad a_j(z) = \sum_{l=0}^{P} d_{j,l} z^l,
\]
where the coefficients $d_{j,l} \in \mathcal{O}_{K,(p)}$ (i.e., are $p$-integers), and are chosen via Siegel's lemma so that the Taylor expansion of $E_P(z)$ has its first $P^2$ coefficients equal to zero. Denote $\Omega^k \alpha = \alpha^{p^k}$.

\subsubsection{Height control of the coefficients $d_{j,l}$.}
By Siegel's lemma, there exists a constant $C_0$ (depending on $K$ and $p$, but independent of $P$ and $k$) such that
\begin{equation}
h(d_{j,l}) \leq C_0 P \quad \text{for all } j,l.
\label{eq:height_coeff}
\end{equation}

\subsubsection{Height of $\alpha^{p^k}$.}
For a fixed algebraic number $\alpha$, we have $h(\alpha^m) = |m| h(\alpha)$. Hence,
\begin{equation}
h(\alpha^{p^k}) = p^k h(\alpha).
\label{eq:height_alpha_pk}
\end{equation}
Consequently, for $l \le P$,
\begin{equation}
h(\alpha^{l p^k}) = l p^k h(\alpha) \le P p^k h(\alpha).
\label{eq:height_alpha_lpk}
\end{equation}

\subsubsection{Iterative height estimate for $T_p(\alpha^{p^k})$.}
Iterating the functional equation $T_p(z^p) = (1 - z^p) T_p(z)^p$ yields
\[
T_p(\alpha^{p^k}) = T_p(\alpha) \prod_{i=1}^{k} (1 - \alpha^{p^i})^{1/p^i}.
\]
Set $\beta_i = (1 - \alpha^{p^i})^{1/p^i}$. Using properties of the height,
\[
h(\beta_i) = \frac{1}{p^i} h(1 - \alpha^{p^i}) \le \frac{1}{p^i} \left( h(\alpha^{p^i}) + \log 2 \right) = \frac{1}{p^i} \left( p^i h(\alpha) + \log 2 \right) = h(\alpha) + \frac{\log 2}{p^i}.
\]
Thus,
\[
h\left( \prod_{i=1}^{k} \beta_i \right) \le \sum_{i=1}^{k} h(\beta_i) \le k h(\alpha) + \log 2 \sum_{i=1}^{k} \frac{1}{p^i} \le k h(\alpha) + \frac{\log 2}{p-1}.
\]
Therefore,
\begin{equation}
h(T_p(\alpha^{p^k})) \le h(T_p(\alpha)) + h\left( \prod_{i=1}^{k} \beta_i \right) \le h(T_p(\alpha)) + k h(\alpha) + \frac{\log 2}{p-1}.
\label{eq:height_Tpk}
\end{equation}
Consequently, there exists a constant $C_1$ (independent of $P$ and $k$) such that $h(T_p(\alpha^{p^k})) \le C_1 k$. Moreover, for $j \le P$,
\begin{equation}
h\left( T_p(\alpha^{p^k})^j \right) = j \, h(T_p(\alpha^{p^k})) \le P \cdot C_1 k = C_1 P k.
\label{eq:height_Tpk_power}
\end{equation}

\subsubsection{Height of a single term $a_j(\alpha^{p^k}) T_p(\alpha^{p^k})^j$.}
We first estimate \ $a_j(\alpha^{p^k}) = \sum_{l=0}^{P} d_{j,l} \alpha^{l p^k}$. Using properties of heights under addition and multiplication:
\begin{itemize}
    \item For each monomial $d_{j,l} \alpha^{l p^k}$, from \eqref{eq:height_coeff} and \eqref{eq:height_alpha_lpk} we obtain
    \[
    h(d_{j,l} \alpha^{l p^k}) \le h(d_{j,l}) + h(\alpha^{l p^k}) \le C_0 P + P p^k h(\alpha).
    \]
    \item Since $a_j(\alpha^{p^k})$ is a sum of $P+1$ such terms, the height of the sum is bounded by the maximum height of its terms plus a logarithmic term:
    \begin{equation}
    h(a_j(\alpha^{p^k})) \le \max_{0 \le l \le P} h(d_{j,l} \alpha^{l p^k}) + \log(P+1) \le C_0 P + P p^k h(\alpha) + \log(P+1).
    \label{eq:height_aj}
    \end{equation}
\end{itemize}
Combining \eqref{eq:height_Tpk_power} and \eqref{eq:height_aj}, the height of the product $a_j(\alpha^{p^k}) T_p(\alpha^{p^k})^j$ satisfies
\begin{equation}
\begin{split}
h\!\left( a_j(\alpha^{p^k})\, T_p(\alpha^{p^k})^j \right)
&\le h\!\left(a_j(\alpha^{p^k})\right)
+ h\!\left( T_p(\alpha^{p^k})^j \right) \\
&\le \left( C_0 P + P p^k h(\alpha) + \log(P+1) \right)
+ C_1 P k.
\end{split}
\label{eq:height_single_term_raw}
\end{equation}
When $k$ is sufficiently large, the term $p^k$ dominates, so there exists a constant $C_2$ (independent of $P$ and $k$) such that
\begin{equation}
h\left( a_j(\alpha^{p^k}) T_p(\alpha^{p^k})^j \right) \le C_2 P p^k \quad \text{for all } j.
\label{eq:height_single_term_bound}
\end{equation}

\subsubsection{Height of the entire sum $E_P(\alpha^{p^k})$.}
Since $E_P(\alpha^{p^k}) = \sum_{j=0}^{P} a_j(\alpha^{p^k}) T_p(\alpha^{p^k})^j$ is a sum of $P+1$ terms, we again apply the height bound for sums:
\[
h\left( E_P(\alpha^{p^k}) \right) \le \max_{0 \le j \le P} h\left( a_j(\alpha^{p^k}) T_p(\alpha^{p^k})^j \right) + \log(P+1).
\]
Substituting \eqref{eq:height_single_term_bound} yields
\[
h\left( E_P(\alpha^{p^k}) \right) \le C_2 P p^k + \log(P+1).
\]
For large $k$, the term $\log(P+1)$ is absorbed, and therefore there exists a constant $c_1$ (depending on $K$, $p$, $h(\alpha)$, and $h(T_p(\alpha))$, but independent of $P$ and $k$) such that
\begin{equation}
h\bigl( E_P(\Omega^k \alpha) \bigr) \le c_1 P p^k.
\label{eq:height_final}
\end{equation}

This upper bound, $h(E_P(\Omega^k \alpha)) = O(P p^k)$, is the crucial arithmetic estimate that, when combined with the much smaller analytic upper bound $|E_P(\Omega^k \alpha)| = O(|\alpha|^{P^2 p^k})$, leads to a contradiction via Liouville's inequality unless $T_p(\alpha)$ is transcendental. The linear growth $O(P p^k)$ in the height estimate is attainable because the $p$-integrality of both the auxiliary coefficients $d_{j,l}$ and the Taylor coefficients $t_p(n)$ eliminates the contributions of all non-$p$ non-Archimedean valuations. Consequently, the height computation reduces to estimating only the Archimedean valuations, which naturally produce the linear factor in $P$.

\subsection{Height Estimates for the Algebraic Independence Proof}

We now establish the height bound required for the proof of Theorem \ref{self-made2}, the algebraic independence of the values $T_p(\alpha_1),\dots,T_p(\alpha_m)$ at multiplicatively independent algebraic points. The strategy parallels the transcendence proof but involves a multivariable auxiliary function.

Let $p$ be a prime, and let $\alpha_1,\dots,\alpha_m \in \overline{\mathbb{Q}}$ be nonzero algebraic numbers with $|\alpha_i|<1$ for all $i$ and assume they are multiplicatively independent. For contradiction, suppose that $T_p(\alpha_1),\dots,T_p(\alpha_m)$ are algebraically dependent over $\overline{\mathbb{Q}}$. Set $K = \mathbb{Q}(\alpha_1,\dots,\alpha_m, T_p(\alpha_1),\dots,T_p(\alpha_m))$, a finite extension of $\mathbb{Q}$.

We construct a multivariable auxiliary function
\[
E_P(z_1,\dots,z_m) = \sum_{0\le j_1,\dots,j_m \le P} a_{j_1,\dots,j_m}(z_1,\dots,z_m) \; T_p(z_1)^{j_1}\cdots T_p(z_m)^{j_m},
\]
where each $a_{\mathbf{j}}(\mathbf{z}) = a_{j_1,\dots,j_m}(z_1,\dots,z_m) \in \mathcal{O}_{K,(p)}[z_1,\dots,z_m]$ is a polynomial of total degree at most $P$ with coefficients in $\mathcal{O}_{K,(p)}$ (i.e., $p$-integers). These polynomials are chosen, via Siegel's lemma, so that the Taylor expansion of $E_P$ around the origin satisfies
\[
E_P(z_1,\dots,z_m) = \sum_{L_1,\dots,L_m \ge 0} b_{L_1,\dots,L_m} z_1^{L_1}\cdots z_m^{L_m},
\]
with
\[
b_{L_1,\dots,L_m}=0 \quad \text{whenever } L_1+\cdots+L_m < P^2.
\]
Denote $\mathbf{\alpha} = (\alpha_1,\dots,\alpha_m)$ and $\mathbf{\Omega}^k\mathbf{\alpha} = (\alpha_1^{p^k},\dots,\alpha_m^{p^k})$.

We aim to prove that there exists a constant $c_1>0$, independent of $P$ and $k$, such that
\begin{equation}
h\!\left(E_P(\mathbf{\Omega}^k\mathbf{\alpha})\right) \le c_1 P p^k.
\label{eq:multi_height_goal}
\end{equation}

\subsubsection{Height control of the polynomial coefficients.}
Applying Siegel's lemma to the linear system that imposes the vanishing conditions, we obtain a nontrivial solution with coefficients in $\mathcal{O}_{K,(p)}$. Moreover, the absolute logarithmic heights of these coefficients satisfy
\begin{equation}
h(\text{coeff. of } a_{\mathbf{j}}) \le C_0 P \quad \text{for all } \mathbf{j},
\label{eq:multi_height_coeff}
\end{equation}
where $C_0$ depends on $K$, $p$, and $m$, but not on $P$ or $k$.

\subsubsection{Heights of the monomials in $\alpha_i^{p^k}$.}
For each $i$, $h(\alpha_i^{p^k}) = p^k h(\alpha_i)$. If $l_i \le P$, then
\begin{equation}
h(\alpha_i^{l_i p^k}) = l_i p^k h(\alpha_i) \le P p^k h(\alpha_i).
\label{eq:multi_height_alpha_monomial}
\end{equation}
For a product $\alpha_1^{l_1 p^k}\cdots\alpha_m^{l_m p^k}$ with $0\le l_i\le P$, we have
\begin{equation}
h\!\left(\alpha_1^{l_1 p^k}\cdots\alpha_m^{l_m p^k}\right) = \sum_{i=1}^m l_i p^k h(\alpha_i) \le P p^k \sum_{i=1}^m h(\alpha_i) = O(P p^k).
\label{eq:multi_height_alpha_product}
\end{equation}

\subsubsection{Heights of the factors $T_p(\alpha_i^{p^k})^{j_i}$.}
From the functional equation, as before,
\[
T_p(\alpha_i^{p^k}) = T_p(\alpha_i) \prod_{\ell=1}^k (1-\alpha_i^{p^\ell})^{1/p^\ell}.
\]
Hence,
\[
h(T_p(\alpha_i^{p^k})) \le h(T_p(\alpha_i)) + k h(\alpha_i) + \frac{\log 2}{p-1} = O(k).
\]
For $j_i \le P$, we obtain
\begin{equation}
h\!\left(T_p(\alpha_i^{p^k})^{j_i}\right) = j_i \, h(T_p(\alpha_i^{p^k})) \le P \cdot O(k) = O(P k).
\label{eq:multi_height_Tp_power}
\end{equation}
Consequently, for a product $\prod_{i=1}^m T_p(\alpha_i^{p^k})^{j_i}$ with $0\le j_i\le P$,
\begin{equation}
h\!\left(\prod_{i=1}^m T_p(\alpha_i^{p^k})^{j_i}\right) \le \sum_{i=1}^m O(P k) = O(P k).
\label{eq:multi_height_Tp_product}
\end{equation}

\subsubsection{Height of a single term $a_{\mathbf{j}}(\mathbf{\Omega}^k\mathbf{\alpha}) \prod_i T_p(\alpha_i^{p^k})^{j_i}$.}
The polynomial $a_{\mathbf{j}}(\mathbf{z})$ has total degree $\le P$. Write it as a sum of monomials:
\[
a_{\mathbf{j}}(\mathbf{z}) = \sum_{\substack{0\le l_1,\dots,l_m\le P \\ l_1+\cdots+l_m \le P}} c_{\mathbf{l}}^{(\mathbf{j})} z_1^{l_1}\cdots z_m^{l_m}, \qquad c_{\mathbf{l}}^{(\mathbf{j})} \in \mathcal{O}_{K,(p)}.
\]
Then
\[
a_{\mathbf{j}}(\mathbf{\Omega}^k\mathbf{\alpha}) = \sum_{\mathbf{l}} c_{\mathbf{l}}^{(\mathbf{j})} \alpha_1^{l_1 p^k}\cdots\alpha_m^{l_m p^k}.
\]
For each summand, using \eqref{eq:multi_height_coeff} and \eqref{eq:multi_height_alpha_product},
\begin{equation}
h\!\left(c_{\mathbf{l}}^{(\mathbf{j})} \alpha_1^{l_1 p^k}\cdots\alpha_m^{l_m p^k}\right) \le h(c_{\mathbf{l}}^{(\mathbf{j})}) + h\!\left(\alpha_1^{l_1 p^k}\cdots\alpha_m^{l_m p^k}\right) \le C_0 P + O(P p^k) = O(P p^k).
\end{equation}
The number of summands is at most $\binom{m+P}{m} = O(P^m)$, a polynomial in $P$. Therefore, the height of the sum is bounded by the maximum height of its terms plus the logarithm of the number of terms:
\begin{equation}
\begin{split}
h\!\left(a_{\mathbf{j}}(\mathbf{\Omega}^k\mathbf{\alpha})\right) &\le \max_{\mathbf{l}} h\!\left(c_{\mathbf{l}}^{(\mathbf{j})} \alpha_1^{l_1 p^k}\cdots\alpha_m^{l_m p^k}\right) + \log\!\left(O(P^m)\right)\\ & \le O(P p^k) + O(\log P) = O(P p^k).
\label{eq:multi_height_aj}
\end{split}
\end{equation}
Now, combining \eqref{eq:multi_height_aj} and \eqref{eq:multi_height_Tp_product}, the height of the full term satisfies
\begin{equation}
\begin{split}
h\!\left(a_{\mathbf{j}}(\mathbf{\Omega}^k\mathbf{\alpha}) \prod_{i=1}^m T_p(\alpha_i^{p^k})^{j_i}\right)& \le h\!\left(a_{\mathbf{j}}(\mathbf{\Omega}^k\mathbf{\alpha})\right) + h\!\left(\prod_{i=1}^m T_p(\alpha_i^{p^k})^{j_i}\right) \\ & \le O(P p^k) + O(P k) = O(P p^k),
\end{split}
\label{eq:multi_height_single_term}
\end{equation}
since $p^k$ dominates $k$ for large $k$.

\subsubsection{Height of the full sum $E_P(\mathbf{\Omega}^k\mathbf{\alpha})$.}
The auxiliary function $E_P(\mathbf{\Omega}^k\mathbf{\alpha})$ is a sum of at most $(P+1)^m$ terms of the form estimated in \eqref{eq:multi_height_single_term}. Using the property that the height of a sum is bounded by the maximum height of its terms plus the logarithm of the number of terms, we obtain
\[
h\!\left(E_P(\mathbf{\Omega}^k\mathbf{\alpha})\right) \le \max_{\mathbf{j}} h\!\left(a_{\mathbf{j}}(\mathbf{\Omega}^k\mathbf{\alpha}) \prod_i T_p(\alpha_i^{p^k})^{j_i}\right) + \log\!\left((P+1)^m\right).
\]
The first term is $O(P p^k)$ by \eqref{eq:multi_height_single_term}, and the second term is $O(\log P)$. Thus, for sufficiently large $k$, there exists a constant $c_1>0$ (depending on $K$, $p$, $m$, and the heights of $\alpha_i$ and $T_p(\alpha_i)$) such that
\[
h\!\left(E_P(\mathbf{\Omega}^k\mathbf{\alpha})\right) \le c_1 P p^k.
\]
This establishes the desired height bound \eqref{eq:multi_height_goal}.

The crucial point, as in the algebraic independence proof, is that the $p$-integrality of the coefficients $a_{\mathbf{j}}$ and of the series $T_p(z)$ ensures that the height of the auxiliary function grows only linearly in $P$ (times $p^k$), whereas its Taylor expansion vanishes to order $P^2$, creating an exponential gap that forces the desired arithmetic conclusion.

\section*{Acknowledgements}
The author is deeply grateful to the referees for their insightful comments and constructive criticism on an earlier version of this manuscript. Their remarks, particularly regarding the motivation and contextualization of this work within the broader literature, were invaluable in shaping the present paper into a significantly improved version.

\section*{Declaration of generative AI and AI-assisted technologies in the manuscript preparation process}
During the preparation of this work the author used ChatGPT-5.1 in order to improve clarity and readability. After using this tool/service, the author reviewed and edited the content as needed and takes full responsibility for the content of the published article.

\bibliographystyle{unsrt}
\bibliography{references}

\end{document}